\documentclass[11pt]{article}

\usepackage[T1]{fontenc}
\usepackage[utf8]{inputenc}
\usepackage[margin=1in]{geometry}
\usepackage{amsmath,amssymb,amsthm,mathtools}
\usepackage{booktabs}
\usepackage{enumitem}
\usepackage{microtype}
\usepackage[colorlinks=true,linkcolor=blue,citecolor=blue,urlcolor=blue]{hyperref}

\newtheorem{theorem}{Theorem}
\newtheorem{proposition}{Proposition}
\newtheorem{fact}{Fact}
\newtheorem{lemma}{Lemma}
\newtheorem{conjecture}{Conjecture}
\newtheorem{problem}{Problem}
\theoremstyle{definition}
\newtheorem*{notation}{Notation}

\DeclareMathOperator{\ex}{ex}
\DeclareMathOperator{\EX}{EX}
\newcommand{\Wge}[1]{W_{\ge #1}}
\newcommand{\Cge}[1]{C_{\ge #1}}
\newcommand{\dG}{d_G}
\newcommand{\Nbr}{N_G}

\title{Counting triangles in graphs with no wheels of order at least five }
\author{
Chunyang Dou\thanks{School of Mathematical Sciences, Anhui University, Hefei 230601, P. R. China. Email: \texttt{chunyang@stu.ahu.edu.cn}. Supported in part by the National Natural Science Foundation of China (No. 12471319) and Anhui Provincial Department of Education Research Project (No.~2025AHGXZK50154).}
\and Bo Ning\thanks{College of Computer Science, Nankai University, Tianjin 300350, P. R. China. Email: \texttt{bo.ning@nankai.edu.cn}. Supported in part by the National Natural Science Foundation of China (No. 12371350) and the Fundamental Research Funds for the Central Universities (No. 63243151).}
\and Xing Peng\thanks{Center for Pure Mathematics, School of Mathematical Sciences, Anhui University, Hefei 230601, P. R. China. Email: \texttt{x2peng@ahu.edu.cn}. Supported in part by the National Natural Science Foundation of China (No. 12471319) and Excellent University Research and Innovation Team in Anhui Province (No. 2025AHGXZK10041).}
}
\date{}

\begin{document}
\maketitle

\begin{abstract}
For a family of graphs $\mathcal F$, a graph $G$ is said to be $\mathcal F$-free if it contains no member of $\mathcal F$ as a subgraph. A wheel graph $W_k$ is a graph on $k+1$ vertices formed by joining a new vertex to all vertices of a $k$-cycle. Given an integer $k\ge 3$, we consider the problem of determining the maximum number of triangles in a $\Wge{k}$-free graph, where $\Wge{k}=\{W_\ell: \ell \geq k\}$. The case $k=3$ was raised by Gallai, who proposed a conjecture for this case (see Erd\H{o}s \cite{E88}). Gallai's conjecture was disproved by Zhou \cite{Zhou} and independently by F\"uredi, Goemans, and Kleitman \cite{FGK}. In this paper, we study the case $k=4$. Namely, for every  integer $n\ge 3$, we determine the maximum number of triangles in an $n$-vertex $\Wge{4}$-free graph and characterize all extremal graphs. We prove that for every integer $n\ge 3$,
$\ex(n,K_3,\Wge{4})=h(n),$
where \[
h(n):=
\begin{cases}
\dfrac{n^2}{4}, & n\equiv 0 \pmod 4,\\[2mm]
\dfrac{n^2}{4}-1, & n\equiv 2 \pmod 4,\\[2mm]
\left\lfloor\dfrac{n^2-n}{4}\right\rfloor, & n\equiv 1,3 \pmod 4.
\end{cases}
\]

\end{abstract}

\section{Introduction}

Tur\'an problems are among the central themes of extremal combinatorics. Given a family of graphs $\mathcal F$, if a graph $G$ does not contain a member of $\mathcal F$ as a subgraph, then we say that $G$ is $\mathcal F$-free. Typically, Tur\'an problems ask, for a given graph $T$ and a family of forbidden subgraphs $\mathcal F$, what is the maximum number of copies of $T$ in an $n$-vertex $\mathcal F$-free graph. This maximum is the generalized Tur\'an number, denoted by $\ex(n,T,\mathcal F)$. If an $\mathcal F$-free graph $G$ contains $\ex(n,T,\mathcal F)$ copies of $T$, then we say that $G$ is an {\it extremal graph}. For $T=K_2$, this is the classical Tur\'an number $\ex(n,\mathcal F)$. When $\mathcal{F}=\{F\}$, we write $\ex(n,T,F)$ and $\ex(n,F)$.

Mantel's theorem, which gives $\ex(n,K_3)$, is perhaps the earliest result in the study of Tur\'an numbers of graphs. Tur\'an's theorem \cite{Turan} famously determines $\ex(n,K_r)$ for all $r$ and is widely regarded as the starting point of extremal graph theory. A landmark result in this area is the Erd\H{o}s--Stone--Simonovits theorem \cite{ES,ES1}, which provides an asymptotic formula for $\ex(n,\mathcal F)$:
\[
\ex(n,\mathcal F)=\left(1-\frac{1}{p(\mathcal F)}+o(1)\right)\binom n2,
\]
where $p(\mathcal F)=\min\{\chi(F)-1:F\in\mathcal F\}$. If $\mathcal F$ contains a bipartite graph, then $p(\mathcal F)=1$ and the family is called degenerate. The Erd\H{o}s--Stone--Simonovits theorem then gives only $\ex(n,\mathcal F)=o(n^2)$, and determining the order of magnitude of $\ex(n,\mathcal F)$ for a degenerate family is generally challenging.

Several papers deal with the function $\ex(n,T,H)$ for $T\ne K_2$. The first is due to Erd\H{o}s \cite{E62}, who determined $\ex(n,K_t,K_r)$ for all $t<r$; see also Bollob\'as \cite{BB} for extensions. Perhaps the most natural subgraph to count, after edges, is the triangle. Let $F_k$ be the graph consisting of $k$ triangles sharing a common vertex. To our knowledge, Erd\H{o}s and S\'os \cite{ESos} initiated the study of counting triangles in $F_k$-free graphs; in particular, they determined $\ex(n,K_3,F_2)$. In an influential paper, Alon and Shikhelman \cite{AS} proved an upper bound on $\ex(n,K_3,F_k)$ for all $k$. Zhu, Chen, Gerbner, Gy\H{o}ri, and Karim \cite{Zhu} obtained the exact value of $\ex(n,K_3,F_k)$ and characterized the extremal graphs for $n\ge 4k^3$.

For a graph $G$, let $t(G)$ denote the number of triangles in $G$. Given a vertex $v\in V(G)$, let $\Nbr(v)=\{u:u\text{ is adjacent to }v\}$ be the {\it neighborhood} of $v$ and let $\dG(v)=|\Nbr(v)|$ be its {\it degree}. We shall repeatedly use the identity
\begin{equation}\label{eq:triangle-neighborhood}
 t(G)=\frac13\sum_{v\in V(G)} e(G[\Nbr(v)]),
\end{equation}
where $e(G[\Nbr(v)])$ is the number of edges in the graph induced by $\Nbr(v)$. Therefore, if one imposes a constraint on the neighborhood of each vertex, then one may obtain an upper bound on $t(G)$. In the study of $\ex(n,K_3,F_k)$, the graph $G[\Nbr(v)]$ is $M_k$-free, where $M_k$ is a matching with $k$ edges. More generally, given a graph $F$, let $\widehat F$ denote the graph obtained from $F$ by joining a new vertex to every vertex of $F$. Then
\[
 \ex(n,K_3,\widehat F)
 \le \frac13\sum_{v\in V(\Gamma)} \ex(d_{\Gamma}(v),F),
\]
 where $\Gamma$ is an extremal graph for $\ex(n,K_3,\widehat F)$.
Thus, in studying $\ex(n,K_3,\widehat F)$, it is natural to choose graphs $F$ for which the Tur\'an number $\ex(n,F)$ is well understood.

Recall that in the influential paper \cite{EG59}, Erd\H{o}s and Gallai determined the Tur\'an number $\ex(n,M_k)$ and proved tight upper bound for $\ex(n,P_k)$ and $\ex(n, C_{\geq k})$. Here $C_{\geq k}$ is the set of cycles with length at least $k$.
Since $\ex(n,K_3,F_k)$ is known for large $n$ and $F_k=\widehat M_k$, one may ask for $\ex(n,K_3,\widehat P_k)$. Mubayi and Mukherjee \cite{MM} first proved the asymptotic value of $\ex(n,K_3,\widehat P_k)$ for $k=4,5,6$. For $k\ge 7$, they conjectured that
\[
\ex(n,K_3,\widehat P_k)=\left\lfloor\frac{k-2}{2}\right\rfloor \frac{n^2}{8}+o(n^2).
\]
This conjecture remains open for $k\geq 7$. Later, Mukherjee \cite{M} determined the exact value of $\ex(n,K_3,\widehat P_4)$ for all $n$. Using a stability theorem for the Tur\'an problem of a non-bipartite graph, Hei, Hou, and Ma \cite{HHM} recently showed the exact value of $\ex(n,K_3,\widehat P_5)$ for sufficiently large $n$.

 A wheel graph $W_k$ is a graph on $k+1$ vertices formed by joining a new vertex to all vertices of a $k$-cycle. Notice that $\widehat C_k$ is the wheel graph $W_k$. Let $\Wge{k}$ be the set of wheels on at least $k+1$ vertices. Following the study of Erd\H{o}s and Gallai \cite{EG59}, it is natural to ask the following general problem.

\begin{problem}\label{prob:main}
For a fixed integer $k\ge 3$, determine the value of $\ex(n,K_3,\Wge{k})$.
\end{problem}

For the case $k=3$, Gallai, as reported by Erd\H{o}s \cite{E88} and independently Zelinka \cite{Zelinka} proposed the following conjecture.

\begin{conjecture}\label{conj:gallai}
For $n\ge 4$, we have
\[
\ex(n,K_3,\Wge{3})=\left\lfloor \frac{n^2}{8}\right\rfloor.
\]
\end{conjecture}

Note that if $G$ contains no wheels, then the neighborhood of each vertex in $G$ is a forest. Zhou \cite{Zhou} disproved Gallai's conjecture by constructing a $\Wge{3}$-free graph with $(n^2+n)/8$ triangles whenever $n$ is of the form $8q+7$. He also showed $\ex(n,K_3,\Wge{3})\le (n^2-n)/6$. Gallai's conjecture was also disproved by F\"uredi, Goemans, and Kleitman \cite{FGK}, who exhibited graphs containing $n^2/7.5$ triangles. Moreover, they proved the upper bound $\ex(n,K_3,\Wge{3})\le n^2/7.02+O(n)$. Gallai's question is notably hard and remains open.

In this paper,
we address the next case of
Problem~\ref{prob:main}, namely $k=4$.  We determine the exact value of $\ex(n,K_3,\Wge{4})$ and characterize all extremal graphs. In a $\Wge{4}$-free graph, the only possible cycle in the neighborhood of a vertex is a triangle.

We now introduce the extremal construction. For positive integers $a$ and $b$, let $H_{a,b}$ be the graph obtained from the complete bipartite graph $K_{a,b}$ by adding a maximum matching in each part. For $n\le 4$, set $H(n)=K_n$. For $n\ge 5$, define
\[
H(n):=
\begin{cases}
H_{2m,2m}, & n=4m,\\
H_{2m,2m+1}, & n=4m+1,\\
H_{2m,2m+2}, & n=4m+2,\\
H_{2m+1,2m+2}, & n=4m+3.
\end{cases}
\]
Observe that the neighborhood of each vertex of $H(n)$ may contain a triangle but contains no  cycle of length at least four. Thus $H(n)$ is $\Wge{4}$-free. Let
\[
h(n):=
\begin{cases}
\dfrac{n^2}{4}, & n\equiv 0 \pmod 4,\\[2mm]
\dfrac{n^2}{4}-1, & n\equiv 2 \pmod 4,\\[2mm]
\left\lfloor\dfrac{n^2-n}{4}\right\rfloor, & n\equiv 1,3 \pmod 4.
\end{cases}
\]
It is straightforward to verify that $t(H(n))=h(n)$. Our main result is the following theorem.

\begin{theorem}\label{main}
For every integer $n\ge 3$,
\[
\ex(n,K_3,\Wge{4})=h(n).
\]
Moreover, $H(n)$ is the unique extremal graph.
\end{theorem}

For $n\le4$, the graph $H(n)$ is complete and the assertion is immediate. Thus the lower bound follows from the construction $H(n)$, and it remains to prove the upper bound $\ex(n,K_3,\Wge{4})\le h(n)$ and the uniqueness of the extremal graph for $n\ge 5$.

\begin{notation}
Let $G$ be a graph. For a vertex $x\in V(G)$ and a set $U\subseteq V(G)$, let $N_U(x)=\Nbr(x)\cap U$ and $d_U(x)=|N_U(x)|$.
\end{notation}

\section{Reduction and proof sketch}\label{sec:sketch}

In this section, we introduce the two statements that imply Theorem~\ref{main}, and then give a sketch for their proofs. This mirrors the structure of the proof: first a deletion argument gives the sharp upper bound, and then a more delicate equality analysis identifies the unique extremal graph.

We shall prove the following two propositions in Section~\ref{sec:proof-main}.

\begin{proposition}[Upper bound]\label{prop:upper}
For every integer $n\ge 1$, every $n$-vertex $\Wge{4}$-free graph $G$ satisfies
\[
 t(G)\le h(n).
\]
\end{proposition}

\begin{proposition}[Extremal graphs]\label{prop:extremal}
Let $n\ge 3$, and let $G$ be an $n$-vertex $\Wge{4}$-free graph with $$t(G)=h(n).$$ Then $G\cong H(n)$.
\end{proposition}

\begin{proof}[Proof of Theorem~\ref{main} (assuming Propositions~\ref{prop:upper} and \ref{prop:extremal})]
The graph $H(n)$ is $\Wge{4}$-free and satisfies $t(H(n))=h(n)$. Thus
$\ex(n,K_3,\Wge{4})\ge h(n).$
Proposition~\ref{prop:upper} gives the reverse inequality. If equality holds, Proposition~\ref{prop:extremal} gives $G\cong H(n)$, and hence the extremal graph is unique.
\end{proof}

\paragraph{Proof sketch.}
The central technical part of the paper is the proof of Proposition~\ref{prop:upper}. Let $G$ be an $n$-vertex $\Wge{4}$-free graph. If $G$ contains no $K_4$, then every neighborhood $G[\Nbr(v)]$ is both $\Cge{4}$-free and triangle-free, hence it is a forest. Combining the  identity \eqref{eq:triangle-neighborhood} and Tur\'an's theorem for $K_4$-free graphs, we  can show $t(G)<h(n)$. Thus any extremal graph must contain a copy $U$ of $K_4$.

Once such a $K_4$ is fixed, we will show every vertex outside $U$ can have at most two neighbors in $U$, and the possible neighborhoods in $K_4$ of adjacent outside vertices are strongly restricted; this is Lemma~\ref{lem:label}. Let $\tau(U)$ be the number of triangles intersecting $U$. The key estimate is the $K_4$-touching bound, Lemma~\ref{lem:touching}, which states that
\[
 \tau(U)\le
 \begin{cases}
 2n-4, & n\text{ even},\\
 2n-5, & n\text{ odd}.
 \end{cases}
\]
By Fact~\ref{fact:diff}, this upper bound is exactly $h(n)-h(n-4)$. Deleting $U$ and applying induction gives
\[
 t(G)=t(G-U)+\tau(U)\le h(n-4)+h(n)-h(n-4)=h(n),
\]
which proves the upper bound.

It remains to discuss the equality case. When $n$ is even, equality in the neighborhood estimate forces $G$ to be an extremal $W_4$-free graph for the classical Tur\'an problem. The  theorem of Dzido and Jastrz\k{e}bski then identifies the graph, and it must be $H(n)$ by counting triangles. When $n$ is odd, the argument is inductive. Equality in the deletion step implies that $G-U\cong H(n-4)$ and there are only a small number of vertices outside $U$ with exactly one neighbor in $U$ (which are called singletons). If there are singletons, then we will show the number of triangles intersecting $U$ in exactly one vertex is strictly less than what we expect as it is related to the number of singletons. This contradiction implies that singletons do not exist and we can show the graph is exactly  $H(n)$.

\section{Preliminaries}\label{sec:prelim}

We start with a simple fact.

\begin{fact}\label{fact:diff}
For every integer $n\ge 5$,
\[
 h(n)-h(n-4)=
 \begin{cases}
 2n-4, & n\text{ even},\\
 2n-5, & n\text{ odd}.
 \end{cases}
\]
\end{fact}

Let $\Cge{k}$ be the set of cycles on at least $k$ vertices. We shall use the following classical result of Erd\H{o}s and Gallai.

\begin{theorem}[Erd\H{o}s--Gallai \cite{EG59}]\label{thm:EG}
For $n\ge k\ge 3$, we have
\[
 \ex(n,\Cge{k})\le \frac{(k-1)(n-1)}{2}.
\]
\end{theorem}

We shall also need the following theorem of Dzido and Jastrz\k{e}bski. Note that $W_4$ is a wheel graph on five vertices.

\begin{theorem}[Dzido--Jastrz\k{e}bski \cite{DJ}]\label{thm:W4}
For every integer $m\ge 4$,
\[
\ex(m,W_4)=
\begin{cases}
\dfrac{m^2}{4}+\dfrac{m}{2}-1, & m\equiv 2\pmod 4,\\[2mm]
\left\lfloor\dfrac{m^2}{4}\right\rfloor+\left\lfloor\dfrac{m}{2}\right\rfloor, & \text{otherwise}.
\end{cases}
\]
Moreover, if $m\not\equiv 2\pmod 4$, then the only extremal graph is $H_{a,b}$ with $a+b=m$ and $|a-b|\le 1$. If $m\equiv 2\pmod 4$, then there are exactly two extremal graphs, namely $H_{m/2,m/2}$ and $H_{m/2-1,m/2+1}$.
\end{theorem}

For a fixed copy of $K_4$, the next lemma records restrictions on the neighborhoods of vertices outside the $K_4$.

\begin{lemma}\label{lem:label}
Let $G$ be an $n$-vertex $\Wge{4}$-free graph. Suppose that $U=\{u_1,u_2,u_3,u_4\}$ induces a copy of $K_4$ in $G$. Then the following hold.
\begin{enumerate}[label=\textup{(\roman*)}]
\item For each $x\in V(G)\setminus U$, we have $d_U(x)\le 2$.
\item If $x,y\in V(G)\setminus U$ satisfy $d_U(x)=d_U(y)=2$ and $xy\in E(G)$, then either $N_U(x)=N_U(y)$ or $N_U(x)=U\setminus N_U(y)$.
\item For fixed $1\le i<j\le 4$, the graph induced by
\[
 X_{ij}=\{x\in V(G)\setminus U:N_U(x)=\{u_i,u_j\}\}
\]
consists of a matching together with some isolated vertices.
\item Suppose $x,y,z\in V(G)\setminus U$ satisfy $d_U(x)=d_U(y)=d_U(z)=2$. If $N_U(x)=N_U(y)$ and $zx,zy\in E(G)$, then $N_U(z)=U\setminus N_U(x)$.
\item If $x\in V(G)\setminus U$ satisfies $N_U(x)=\{u_i\}$, then there is at most one vertex $y\in \Nbr(x)\setminus U$ such that $d_U(y)=2$ and $u_i\in N_U(y)$.
\end{enumerate}
\end{lemma}

\begin{proof}
(i) Suppose that some $x\in V(G)\setminus U$ has at least three neighbors in $U$, say $\{u_1,u_2,u_3\}\subseteq N_U(x)$. Then $u_2,u_4,u_3,x$ form a $4$-cycle in $G[\Nbr(u_1)]$, giving a copy of $W_4$, a contradiction.

(ii) Suppose $N_U(x)$ and $N_U(y)$ are neither equal nor complementary. Then, up to relabelling, $N_U(x)=\{u_1,u_2\}$ and $N_U(y)=\{u_1,u_3\}$. Since $xy\in E(G)$, the vertices $x,u_2,u_3,y$ form a $4$-cycle in $G[\Nbr(u_1)]$, a contradiction.

(iii) Suppose that some $z\in X_{ij}$ is adjacent to two distinct vertices $x,y\in X_{ij}$. Then $x,z,y,u_j$ form a $4$-cycle in $G[\Nbr(u_i)]$, a contradiction.

(iv) This follows directly from (ii) and (iii).

(v) Suppose that there are two vertices $y,z\in \Nbr(x)\setminus U$ such that $d_U(y)=d_U(z)=2$ and $u_i\in N_U(y)\cap N_U(z)$. If $N_U(y)=N_U(z)=\{u_i,u_j\}$, then $u_j,y,x,z$ form a $4$-cycle in $G[\Nbr(u_i)]$. If $N_U(y)=\{u_i,u_j\}$ and $N_U(z)=\{u_i,u_k\}$ with $j\ne k$, then $y,x,z,u_k,u_j$ form a $5$-cycle in $G[\Nbr(u_i)]$. In both cases we obtain a forbidden wheel, a contradiction.
\end{proof}

\begin{lemma}\label{lem:noK4}
Let $G$ be a $\Wge{4}$-free graph with $n\ge 5$ vertices. If $G$ does not contain $K_4$, then $t(G)<h(n)$.
\end{lemma}

\begin{proof}
There is nothing to prove if $t(G)=0$, so assume $t(G)>0$. Let $X$ be the set of non-isolated vertices of $G$. Then $|X|\ge 3$. For every $v\in X$, the graph $G[\Nbr(v)]$ is $\Cge{4}$-free because $G$ is $\Wge{4}$-free; moreover, it is triangle-free because $G$ has no $K_4$. Hence $G[\Nbr(v)]$ is a forest and
\[
 e(G[\Nbr(v)])\le \dG(v)-1.
\]
Summing over all $v\in X$ gives
\[
3t(G)=\sum_{v\in X} e(G[\Nbr(v)])
\le \sum_{v\in X}(\dG(v)-1)
=2e(G)-|X|\le 2e(G)-3.
\]
Since $G$ is $K_4$-free, Tur\'an's theorem gives $e(G)\le \lfloor n^2/3\rfloor$. Therefore
\[
 3t(G)\le 2\left\lfloor\frac{n^2}{3}\right\rfloor-3.
\]
If $n$ is even, then
\[
2\left\lfloor\frac{n^2}{3}\right\rfloor-3\le \frac{2n^2}{3}-3<\frac{3n^2}{4}-3\le 3h(n).
\]
If $n\equiv 1\pmod 4$, then
\[
2\left\lfloor\frac{n^2}{3}\right\rfloor-3\le \frac{2n^2}{3}-3<\frac{3(n^2-n)}{4}=3h(n),
\]
because $n^2-9n+36>0$ for all integers $n$. If $n\equiv 3\pmod 4$, then $n\ge 7$ and
\[
2\left\lfloor\frac{n^2}{3}\right\rfloor-3\le \frac{2n^2}{3}-3<\frac{3(n^2-n-2)}{4}=3h(n),
\]
because $n^2-9n+18>0$ for all integers $n\ge 7$. Thus $t(G)<h(n)$ for $n\ge5$.
\end{proof}

\begin{lemma}\label{lem:touching}
Let $G$ be an $n$-vertex $\Wge{4}$-free graph. Suppose that $U=\{u_1,u_2,u_3,u_4\}$ induces a copy of $K_4$. Let $\tau(U)$ denote the number of triangles of $G$ intersecting $U$. Then
\[
 \tau(U)\le
 \begin{cases}
 2n-4, & n\text{ even},\\
 2n-5, & n\text{ odd}.
 \end{cases}
\]
\end{lemma}

\begin{proof}
For each $x\in V(G)\setminus U$, Lemma~\ref{lem:label}(i) gives $d_U(x)\le 2$. For $0\le i\le 2$, let
\[
 \alpha_i=|\{x\in V(G)\setminus U:d_U(x)=i\}|.
\]
Thus $\alpha_0+\alpha_1+\alpha_2=n-4$. Similarly, for $1\le i\le 2$, let $q_i$ be the number of triangles that intersect $U$ in exactly $i$ vertices. Together with the four triangles contained in $U$, we have
\[
 \tau(U)=4+q_1+q_2.
\]
A vertex $x\notin U$ with two neighbors in $U$ determines exactly one triangle intersecting $U$ in two vertices. Hence $q_2=\alpha_2$.

For each $1\le i\le 4$, let
\[
 s_i=|\{x\in V(G)\setminus U:d_U(x)=1,\ N_U(x)=\{u_i\}\}|
\]
and
\[
 p_i=|\{x\in V(G)\setminus U:d_U(x)=2,\ u_i\in N_U(x)\}|.
\]
Then $\sum_{i=1}^4 s_i=\alpha_1$, $\sum_{i=1}^4 p_i=2\alpha_2$, and $\dG(u_i)=3+s_i+p_i$. Let $t_i$ be the number of triangles containing $u_i$ and no other vertex of $U$. Counting edges in $G[\Nbr(u_i)]$, we see that there are three edges contained in $U\setminus\{u_i\}$, exactly $p_i$ edges joining a vertex of $U\setminus\{u_i\}$ to a vertex outside $U$, and exactly $t_i$ edges with both ends outside $U$. Thus,
$e(G[\Nbr(u_i)])=3+p_i+t_i.$
For each $i$, the induced subgraph $G[\Nbr(u_i)]$ is $C_{\ge 4}$-free and $\dG(u_i)\ge 3$. If $\dG(u_i)=3$, then $e(G[\Nbr(u_i)])\le 3(\dG(u_i)-1)/2$. If $\dG(u_i)\ge 4$, then it follows from Theorem \ref{thm:EG}. Thus
\[
3+p_i+t_i=e(G[\Nbr(u_i)])\le \frac{3(\dG(u_i)-1)}2
=\frac{3(2+s_i+p_i)}2.
\]
It follows that
$t_i\le \frac{3s_i}{2}+\frac{p_i}{2}.$
Summing over $1\le i\le 4$, we obtain
\[
 q_1=\sum_{i=1}^4 t_i\le \frac{3\alpha_1}{2}+\alpha_2.
\]
Consequently,
\begin{equation}\label{eq:key-touching}
 \tau(U)=4+q_1+q_2
 \le 4+\frac{3\alpha_1}{2}+2\alpha_2
 =2n-4-2\alpha_0-\frac{\alpha_1}{2}.
\end{equation}
This gives $\tau(U)\le 2n-4$, which proves the desired bound when $n$ is even. If $n$ is odd and $\alpha_0+\alpha_1>0$, then the right-hand side of \eqref{eq:key-touching} is strictly less than $2n-4$, and since $\tau(U)$ is an integer, $\tau(U)\le 2n-5$.

It remains to consider the case in which $n$ is odd and $\alpha_0=\alpha_1=0$. Then every vertex outside $U$ has exactly two neighbors in $U$. For $1\le i<j\le 4$, set
\[
 X_{ij}=\{x\in V(G)\setminus U:N_U(x)=\{u_i,u_j\}\}.
\]
These six sets partition $V(G)\setminus U$. We first claim that if $\{i,j\}\cap\{p,q\}\ne\emptyset$ and $\{i,j\}\ne\{p,q\}$, then $e(X_{ij},X_{pq})=0$. Indeed, if $x\in X_{ij}$, $y\in X_{ik}$ with $j\ne k$, and $xy\in E(G)$, then $x,u_j,u_k,y$ form a $4$-cycle in $G[\Nbr(u_i)]$, a contradiction.

By Lemma~\ref{lem:label}(iii), each $G[X_{ij}]$ consists of a matching together with isolated vertices. Now an edge in $V(G)\setminus U$ contributes to $q_1$ exactly when its endpoints have a common neighbor in $U$. By the preceding claim and Lemma~\ref{lem:label}(iii), this occurs only for edges inside a set $X_{ij}$, and each such edge contributes exactly two triangles to $q_1$. Therefore
\[
 q_1=2\sum_{1\le i<j\le 4} e(G[X_{ij}])
 \le 2\sum_{1\le i<j\le 4}\left\lfloor\frac{|X_{ij}|}{2}\right\rfloor.
\]
Since
\[
 \sum_{1\le i<j\le 4}|X_{ij}|=n-4
\]
is odd, at least one set $X_{ij}$ has odd size. Hence $q_1\le n-5$. Together with $q_2=\alpha_2=n-4$, this yields
\[
 \tau(U)=4+q_1+q_2\le 4+(n-5)+(n-4)=2n-5.
\]
The proof is complete.
\end{proof}

\section{Proof of Theorem~\ref{main}}\label{sec:proof-main}

\subsection{The upper bound}\label{subsec:upper}

\begin{proof}[Proof of Proposition~\ref{prop:upper}]
We proceed by induction on $n$. The cases $n\le 4$ are immediate. Let $n\ge 5$, and assume the statement holds for all smaller orders. Let $G$ be an $n$-vertex $\Wge{4}$-free graph.

If $G$ is $K_4$-free, then Lemma~\ref{lem:noK4} gives $t(G)<h(n)$. Thus we may assume that $G$ contains a copy of $K_4$ with vertex set $U$. Now
\[
 t(G)=t(G-U)+\tau(U).
\]
Since $G-U$ is still $\Wge{4}$-free, the induction hypothesis, Fact~\ref{fact:diff}, and Lemma~\ref{lem:touching} give
\[
 t(G)\le h(n-4)+(h(n)-h(n-4))=h(n).
\]
This proves the upper bound.
\end{proof}

\subsection{The extremal graph for even n}\label{subsec:even}

In this subsection we characterize the extremal graph for even $n$.

\begin{proof}[Proof of Proposition~\ref{prop:extremal} for even $n$]
Let $G$ be an $n$-vertex $\Wge{4}$-free graph with $t(G)=h(n)$, where $n$ is even. First, we claim
 that $G$ has no isolated vertex. Otherwise, assume that $z$ is an isolated vertex. By Proposition \ref{prop:upper}, we have
 \[
 t(G)=t(G-z)\le h(n-1)<h(n).
 \]
 This contradicts $t(G)=h(n)$. Thus $\dG(v)\ge 1$ for every $v$. If $\dG(v)\le 3$, then $ e(G[\Nbr(v)])\le 3(\dG(v)-1)/2$. If $\dG(v)\ge 4$, then we also have  $e(G[\Nbr(v)])\le 3(\dG(v)-1)/2$ by Theorem~\ref{thm:EG} as the graph $G[\Nbr(v)]$ is $\Cge{4}$-free. Therefore,
 for every vertex $v\in V(G)$,  the number of edges in the induced graph $G[\Nbr(v)]$ satisfies
\[
 e(G[\Nbr(v)])\le \frac{3(\dG(v)-1)}2.
\]
Summing over all vertices of $G$, we get
\[
 3t(G)=\sum_{v\in V(G)} e(G[\Nbr(v)])
 \le \frac32\sum_{v\in V(G)}(\dG(v)-1)
 =3e(G)-\frac{3n}{2}.
\]
Since $t(G)=h(n)$, it follows that
\[
 h(n)=t(G)\le e(G)-\frac n2.
\]
Thus
\[
 e(G)\ge h(n)+\frac n2=\ex(n,W_4).
\]
On the other hand, $G$ is $W_4$-free, so $e(G)\le \ex(n,W_4)$. Therefore $G$ is an extremal graph for $\ex(n,W_4)$.

If $n=4q$, then Theorem~\ref{thm:W4} implies that the only extremal graph for $W_4$ is $H_{2q,2q}=H(n)$. If $n=4q+2$, then the two extremal graphs for $W_4$ are $H_{2q+1,2q+1}$ and $H_{2q,2q+2}$. Their numbers of triangles are
$t(H_{2q+1,2q+1})=4q^2+2q<h(n)$
and
$t(H_{2q,2q+2})=4q^2+4q=h(n).$
Since $t(G)=h(n)$, we must have $G\cong H_{2q,2q+2}=H(n)$.
\end{proof}

\subsection{The extremal graph for odd n}\label{subsec:odd}

Let $G$ be a $\Wge{4}$-free graph with $t(G)=h(n)$, where $n$ is odd. We prove that $G\cong H(n)$. This is trivial for $n=3$, so assume $n\ge 5$. The proof is by induction on $n$; we treat $n=5,7,9,11$ and $n\ge 13$ separately.

Throughout this subsection, fix the following setup. Since $t(G)=h(n)$, Lemma~\ref{lem:noK4} implies that $G$ contains a copy $U=\{u_1,u_2,u_3,u_4\}$ of $K_4$. For $x\in V(G)\setminus U$, we call $x$ a \emph{singleton} if $|N_U(x)|=1$ and a \emph{double} if $|N_U(x)|=2$. The set $N_U(x)$ is called the label of $x$. If $N_U(x)=\{u_i,u_j\}$, then write $P=\{u_i,u_j\}$, $P^c=U\setminus P$, and $X_P=X_{ij}$.

By Lemma~\ref{lem:touching} and Fact~\ref{fact:diff},
\[
 h(n)=t(G)\le h(n-4)+\tau(U)\le h(n-4)+2n-5=h(n).
\]
Hence equality holds throughout. In particular,
$\tau(U)=2n-5$ and $t(G')=h(n-4),$
where $G'=G-U$. From \eqref{eq:key-touching}, equality implies
\begin{equation}\label{eq:alpha}
 \alpha_0=0\quad\text{and}\quad \alpha_1\in\{0,1,2\},
\end{equation}
where $\alpha_i=|\{x\in V(G)\setminus U:d_U(x)=i\}|$. Since $\tau(U)=4+q_1+q_2$, $q_2=\alpha_2$, and $n-4=\alpha_1+\alpha_2$, we have
\begin{equation}\label{eq:T1}
 q_1=\tau(U)-4-q_2=2n-9-\alpha_2=n-5+\alpha_1.
\end{equation}
Moreover, by the induction hypothesis, $G'\cong H(n-4)$ for $n\ge 7$. When $n\ge 7$, we may write $G'\cong H_{a,b}$, where
\[
 a=\frac{n-5}{2},\qquad b=\frac{n-3}{2}.
\]

We need one additional  lemma.

\begin{lemma}\label{lem:bip-label}
Let $G$ be a $\Wge{4}$-free graph. Suppose that $U\subseteq V(G)$ induces a copy of $K_4$ and  $V(G)\setminus U$ contains a complete bipartite graph with parts $A$ and $B$.
\begin{enumerate}[label=\textup{(\roman*)}]
\item If there are at least two doubles in one part with the same label $P\subseteq U$, $|P|=2$, then the label of each vertex in the opposite part is contained in $P^c$.
\item If $A$ contains at least two doubles and $B$ contains at least three doubles, then there exists a unique $2$-set $P\subseteq U$ such that all doubles in $A$ have label $P$ and all doubles in $B$ have label $P^c$.
\end{enumerate}
\end{lemma}

\begin{proof}
(i) Let $a_1$ and $a_2$ be two doubles in one part, say in $A$, both with label $P$. Let $b\in B$. Since $A$ and $B$ form a complete bipartite graph, $b$ is adjacent to both $a_1$ and $a_2$. If $b$ is a double, then Lemma~\ref{lem:label}(iv) gives $N_U(b)=P^c$. If $b$ is a singleton, then Lemma~\ref{lem:label}(v) implies that its unique label cannot lie in $P$, so $N_U(b)\subseteq P^c$.

(ii) Choose a double $x\in A$ and write $N_U(x)=P$. Every double in $B$ is adjacent to $x$, so by Lemma~\ref{lem:label}(ii) its label is either $P$ or $P^c$.  Note that $x$ is adjacent to every vertex of $B$. Since $B$ contains at least three doubles and $G[X_P]$ is a matching by Lemma~\ref{lem:label}(iii), at least two doubles in $B$ have label $P^c$. Applying (i) to these two doubles shows that every double in $A$ has label $P$. Since $A$ contains at least two doubles, another application of (i) shows that every double in $B$ has label $P^c$. The uniqueness of $P$ is immediate.
\end{proof}

\medskip
\newcommand{\thickunderline}[2][2pt]{%
  \setbox0=\hbox{#2}%
  \vtop{\hbox{#2}\vspace{0.2ex}\hrule height #1}%
}
\noindent\thickunderline[1pt]{\textbf{The case $n=5$.}}
The graph $G'$ contains only one vertex, say $x$. Thus $q_1=0$. Since $\tau(U)=2n-5=5=4+q_1+q_2$, we get $q_2=1$, which means that $|N_U(x)|=2$. Hence $G\cong H_{2,3}=H(5)$.

\medskip
\providecommand{\thickunderline}{}
\renewcommand{\thickunderline}[2][2pt]{%
  \setbox0=\hbox{#2}%
  \vtop{\hbox{#2}\vspace{0.2ex}\hrule height #1}%
}
\noindent\thickunderline[1pt]
{\textbf{The case $n=7$.}}
The graph $G'$ has three vertices and $h(3)=1$, so $G'$ is a triangle. Write $V(G')=\{x,y,z\}$. By \eqref{eq:T1}, $q_1=2+\alpha_1$.

If $\alpha_1=0$, then all three vertices are doubles. Since $x,y,z$ are pairwise adjacent, Lemma~\ref{lem:label}(ii) and (iii) imply that exactly two of them have the same label and the remaining one has the complementary label. It follows directly that $G\cong H(7)$.

If $\alpha_1=1$, assume that $x$ is the singleton and $y,z$ are doubles. Let $N_U(x)=\{u_1\}$. Then $q_1=3$. If $u_1\notin N_U(y)\cup N_U(z)$, then $q_1\le 2$, a contradiction. Thus, without loss of generality, $u_1\in N_U(y)$. Lemma~\ref{lem:label}(v) implies that $u_1\notin N_U(z)$. Since $y$ and $z$ are adjacent doubles, Lemma~\ref{lem:label}(ii) gives $N_U(y)\cap N_U(z)=\emptyset$. Hence $q_1=1$, again a contradiction.

If $\alpha_1=2$, then $q_1=4$. Each edge of the triangle $xyz$ lies in at most one triangle counted by $q_1$, because each such edge contains a singleton endpoint. Thus $q_1\le 3$, a contradiction. This completes the case $n=7$.

\medskip
\providecommand{\thickunderline}{}
\renewcommand{\thickunderline}[2][2pt]{%
  \setbox0=\hbox{#2}%
  \vtop{\hbox{#2}\vspace{0.2ex}\hrule height #1}%
}
\noindent\thickunderline[1pt]
{\textbf{The case $n=9$.}}
Here $G'\cong H(5)$. We may write $V(G')=A\cup B$, where
$A=\{a_1,a_2\},$ $B=\{b_1,b_2,b_3\},$
$A$ is completely adjacent to $B$. In addition, $A$ induces the edge $a_1a_2$ and the only edge in $B$ is $b_1b_2$.

We first claim that $b_3$ is a double for $U$. Let $U'=\{a_1,a_2,b_1,b_2\}$. Then $G[U']$ is a copy of $K_4$. By the induction hypothesis, $G-U'\cong H(5)$. The graph $H(5)$ contains a unique $K_4$, and every vertex outside this $K_4$ has exactly two neighbors in it. Since $G[U]$ is a $K_4$ in $G-U'$, the remaining vertex $b_3$ is a double for $U$.

By \eqref{eq:T1}, $q_1=4+\alpha_1$. We distinguish three cases.

\smallskip
\noindent\underline{\emph{Case 1: $\alpha_1=0$.}}

All vertices in $A\cup B$ are doubles. By Lemma~\ref{lem:bip-label}(ii), there is a unique $2$-set $P\subseteq U$ such that all vertices in $A$ have label $P$ and all vertices in $B$ have label $P^c$. After relabelling, assume $P=\{u_1,u_2\}$ and $P^c=\{u_3,u_4\}$. Then
$A\cup\{u_3,u_4\}$ and $B\cup\{u_1,u_2\}$
form the two parts of $H_{4,5}$. Thus $G\cong H(9)$.

\smallskip
\noindent\underline{\emph{Case 2: $\alpha_1=1$.}}

We show that this is impossible, since then $q_1=5$.

If the singleton lies in $A$, say $a_1$ is the singleton, then $a_2$ is a double; write $N_U(a_2)=P$. Every vertex of $B$ is a double and is adjacent to $a_2$, so its label is either $P$ or $P^c$. Since $|B|=3$, Lemma~\ref{lem:label}(iii) implies that at least two vertices of $B$ have label $P^c$. Lemma~\ref{lem:bip-label}(i) gives $N_U(a_1)\subseteq P$. Lemma~\ref{lem:label}(v) then implies that no vertex of $B$ has label $P$, so every vertex of $B$ has label $P^c$. The crossing edges contribute nothing to $q_1$, the edge inside $A$ contributes one, and the edge inside $B$ contributes two. Hence $q_1=3$, a contradiction.

If the singleton lies in $B$, say $b_1$ is the singleton, then $b_2$ and $b_3$ are doubles. If $N_U(a_1)=N_U(a_2)=P$, then Lemma~\ref{lem:bip-label}(i) gives $N_U(b_1)\subseteq P^c$ and $N_U(b_2)=N_U(b_3)=P^c$. Then only the edge $a_1a_2$ and the edge $b_1b_2$ contribute to $q_1$, so $q_1\le 3$, a contradiction. Thus, after relabelling, $N_U(a_1)=P$ and $N_U(a_2)=P^c$. We may assume $N_U(b_1)=\{u_1\}\subseteq P$. Lemma~\ref{lem:label}(v) gives $N_U(b_2)=P^c$, while Lemma~\ref{lem:label}(iii) gives $N_U(b_3)=P$. Now $u_1,b_1,a_2,b_3$ form a $4$-cycle in $G[\Nbr(a_1)]$, a contradiction.

\smallskip
\noindent\underline{\emph{Case 3: $\alpha_1=2$.}}

Here $q_1=6$. If both singletons lie in $A$, then all three vertices of $B$ are doubles. By Lemma~\ref{lem:label}(v), the six crossing edges contribute at most two to $q_1$. The edge inside $A$ contributes at most one, and the edge inside $B$ contributes at most two. Thus $q_1\le 5$, a contradiction.

Suppose exactly one singleton lies in each side. We may assume that $a_1\in A$ and $b_1\in B$ are singletons. Let $P=N_U(a_2)$. If $N_U(a_1)\subseteq P$, then Lemma~\ref{lem:label}(v) gives $N_U(b_2)=N_U(b_3)=P^c$. Thus the only edges that may contribute to $q_1$ are $a_1a_2$, $a_1b_1$, $b_1b_2$, and $b_1a_2$, so $q_1\le 4$, a contradiction. If $N_U(a_1)\subseteq P^c$, then one of $b_2,b_3$ has label $P$ and the other has label $P^c$. If $N_U(b_1)\subseteq P$, then the edges incident with $a_1$ contribute one to $q_1$ and all other edges contribute at most four. If $N_U(b_1)\subseteq P^c$, then the edges incident with $a_1$ contribute two and all other edges contribute at most three. In both cases $q_1\le 5$, a contradiction.

Finally, suppose both singletons lie in $B$, say $b_1,b_2$ are singletons. Then $b_3$ is a double and $a_1,a_2$ are doubles. If $a_1$ and $a_2$ have the same label $P$, then every vertex in $B$ has label contained in $P^c$ by Lemma~\ref{lem:bip-label}(i). Hence the crossing edges contribute nothing to $q_1$, the edge $a_1a_2$ contributes two, and the edge $b_1b_2$ contributes at most one. Thus $q_1\le 3$, a contradiction. If the labels of $a_1$ and $a_2$ are complementary, then $b_3$ has one of these two labels. The edges incident with $b_3$ contribute two to $q_1$, each of $b_1$ and $b_2$ contributes at most one through crossing edges, and the edge $b_1b_2$ contributes at most one. Thus $q_1\le 5$, a contradiction. This completes the case $n=9$.

\medskip
\providecommand{\thickunderline}{}
\renewcommand{\thickunderline}[2][2pt]{%
  \setbox0=\hbox{#2}%
  \vtop{\hbox{#2}\vspace{0.2ex}\hrule height #1}%
}
\noindent\thickunderline[1pt]
{\textbf{The case $n=11$.}}
Here $G'\cong H(7)$. We may write $V(G')=A\cup B$, where
$A=\{a_1,a_2,a_3\},$ $B=\{b_1,b_2,b_3,b_4\},$
$A$ and $B$ form a complete bipartite graph, the only edge in $A$ is $a_1a_2$, and the two edges in $B$ are $b_1b_2$ and $b_3b_4$.

\begin{lemma}\label{lem:n11-double}
Every vertex in $\{a_3,b_1,b_2,b_3,b_4\}$ is a double for $U$.
\end{lemma}

\begin{proof}
Let
$U'=\{a_1,a_2,b_1,b_2\},$ $U''=\{a_1,a_2,b_3,b_4\}.$
Both $G[U']$ and $G[U'']$  are copies of $K_4$. By the induction hypothesis, $G-U'\cong H(7)$ and $G-U''\cong H(7)$. Therefore both $G[\{a_3,b_3,b_4\}\cup U]$ and $G[\{a_3,b_1,b_2\}\cup U]$ are isomorphic to $H(7)$. In $H(7)$, every vertex outside a copy of $K_4$ has exactly two neighbors in that $K_4$. This proves the claim.
\end{proof}

Now $q_1=6+\alpha_1$ by \eqref{eq:T1}. Again we distinguish three cases.

\smallskip
\noindent\emph{Case 1: $\alpha_1=0$.} All vertices in $G'$ are doubles. By Lemma~\ref{lem:bip-label}(ii), there exists a $2$-set $P\subseteq U$ such that every vertex of $A$ has label $P$ and every vertex of $B$ has label $P^c$. After relabeling, assume $P=\{u_1,u_2\}$ and $P^c=\{u_3,u_4\}$. Then
$A\cup\{u_3,u_4\}$ and $B\cup\{u_1,u_2\}$
form the two parts of $H_{5,6}$. Thus $G\cong H(11)$.

\smallskip
\noindent\emph{Case 2: $\alpha_1=1$.} Then $q_1=7$. By Lemma~\ref{lem:n11-double}, the singleton lies in $A$. Thus $A$ contains two doubles and $B$ contains four doubles. Lemma~\ref{lem:bip-label}(ii) gives a $2$-set $P\subseteq U$ such that every double in $A$ has label $P$ and every double in $B$ has label $P^c$. The label of the singleton in $A$ is contained in $P$ by Lemma~\ref{lem:bip-label}(i). Hence the crossing edges contribute nothing to $q_1$, the unique edge in $A$ contributes at most one, and the two edges in $B$ contribute four. Thus $q_1\le 5$, a contradiction.

\smallskip
\noindent\emph{Case 3: $\alpha_1=2$.} Then $q_1=8$. By Lemma~\ref{lem:n11-double}, the two singletons are $a_1$ and $a_2$. The vertex $a_3$ is a double, and $B$ contains four doubles. Let $N_U(a_3)=P$. By Lemma~\ref{lem:label}(iii), at least three doubles in $B$ have label $P^c$. Lemma~\ref{lem:bip-label}(i) implies that the labels of all vertices in $A$ are contained in $P$. We may assume $N_U(b_1)=N_U(b_2)=N_U(b_3)=P^c$, and consider the label of $b_4$.

If $N_U(b_4)=P^c$, then the crossing edges contribute nothing to $q_1$. The only edge in $A$, namely $a_1a_2$, contributes at most one, while the two edges inside $B$ contribute four. Hence $q_1\le 5$, a contradiction.

If $N_U(b_4)=P$, then the crossing edges contribute four to $q_1$ in total. The edge $b_1b_2$ contributes two, the edge $b_3b_4$ contributes zero, and the edge $a_1a_2$ contributes at most one. Thus $q_1\le 7$, a contradiction. This completes the case $n=11$.

\medskip
\providecommand{\thickunderline}{}
\renewcommand{\thickunderline}[2][2pt]{%
  \setbox0=\hbox{#2}%
  \vtop{\hbox{#2}\vspace{0.2ex}\hrule height #1}%
}
\noindent\thickunderline[1pt]
{\textbf{The case $n\ge 13$.}}
Recall that $G'\cong H_{a,b}$, where
\[
 a=\frac{n-5}{2}\ge 4,
 \qquad
 b=\frac{n-3}{2}\ge 5.
\]
Let $A$ and $B$ be the two parts of $H_{a,b}$, with $|A|=a$ and $|B|=b$. Let
$D=\{x\in V(G)\setminus U:d_U(x)=2\}$
be the set of doubles. By \eqref{eq:alpha},
$|A\cap D|\ge a-2\ge 2,$
$|B\cap D|\ge b-2\ge 3.$
Choose $x\in A\cap D$ and set $P=N_U(x)$ and $P^c=U\setminus P$. For each $y\in B\cap D$, since $xy\in E(G)$, Lemma~\ref{lem:label}(ii) implies that $y$ has label either $P$ or $P^c$. For two disjoint subsets $L$ and $R$ of $V(G)$, the set of edges between $L$ and $R$ in $G$ is denoted by $E_G(L,R)$.

\begin{lemma}\label{lem:no-wrong-side}
We have $E_G(A,P^c)=\emptyset$ and $E_G(B,P)=\emptyset$.
\end{lemma}

\begin{proof}
By Lemma~\ref{lem:bip-label}(ii), every double in $A$ has label $P$ and every double in $B$ has label $P^c$. Since $A\cap D$ has at least two vertices and $B\cap D$ has at least three vertices, Lemma~\ref{lem:bip-label}(i) also implies that the label of each vertex in $A$ is contained in $P$ and the label of each vertex in $B$ is contained in $P^c$. Thus no vertex of $A$ is adjacent to a vertex of $P^c$, and no vertex of $B$ is adjacent to a vertex of $P$.
\end{proof}

Now consider a triangle $T=\{u_i,v,w\}$ intersecting $U$ in exactly one vertex $u_i$. Lemma~\ref{lem:no-wrong-side} implies that either $\{v,w\}\subseteq A$ or $\{v,w\}\subseteq B$. Since the total number of edges inside $A$ and inside $B$ is
\[
 \left\lfloor\frac a2\right\rfloor+\left\lfloor\frac b2\right\rfloor=\frac{n-5}{2},
\]
and each such edge lies in at most two triangles counted by $q_1$, we obtain
\begin{equation}\label{eq:T1-bound}
 q_1\le n-5.
\end{equation}
Comparing \eqref{eq:T1-bound} with \eqref{eq:T1}, we get $\alpha_1=0$. Thus there are no singletons. Consequently, every vertex of $A$ is a double with label $P$, and every vertex of $B$ is a double with label $P^c$.

It follows that $G$ is obtained from the complete bipartite graph with parts
$A\cup P^c$ and $B\cup P$
by adding a maximum matching inside each part. Therefore
$G\cong H_{a+2,b+2}=H(n).$
This completes the induction for odd $n$ and hence the proof of Proposition~\ref{prop:extremal}.

\section{Concluding remarks}
In this paper, we determine the exact value of $\ex(n, K_3, \Wge{4})$ and characterize all extremal graphs for every integer $n\ge 3$. A natural extension is to consider $\ex(n, K_3, \Wge{k})$ for $k\ge 5$.

Let $\EX(n,T,\mathcal{F})$ be the set of $n$-vertex $\mathcal{F}$-free graphs containing $\ex(n,T,\mathcal{F})$ copies of $T$. When $T=K_2$, we simply write $\EX(n,\mathcal{F})$. Combining Theorems \ref{main} and \ref{thm:W4}, we note that $\EX(n,K_3,\Wge{4})\subseteq\EX(n,W_{4})$. A natural question is whether $\EX(n,K_3,\Wge{k})\subseteq\EX(n,W_{k})$ holds for $k\ge5$ and $n$ sufficiently large. We discuss this problem according to the parity of $k$.

For the case $k = 2s+1\ge 5$, the answer is negative. Let $T(n,3)$ be the complete balanced $3$-partite graph with $n$ vertices. We first note that $\ex(n,W_{2s+1}) = e(T(n,3))$ and  the only extremal graph is $T(n,3)$ for sufficiently large $n$, as $W_{2s+1}$ is a color-critical graph. Thus we have
\[
\EX(n,K_3,\Wge{2s+1})\not \subseteq\EX(n,W_{2s+1})=\{T(n,3)\}\quad\text{and}\quad\ex(n,\Wge{2s+1})<\ex(n,W_{2s+1})
\]
for  sufficiently large $n$, as $T(n,3)$ contains $W_{2s+2}$.  It is  of interest to investigate the Tur\'an number and the generalized Tur\'an number for $\Wge{2s+1}$.

For the case $k = 2s\ge 6$, we introduce the following definition. A graph is called {\em nearly $(s-1)$-regular} if all vertices except one have degree $s-1$, and the remaining vertex has degree $s-2$. Let $\mathcal{U}_{s,n}$ denote the family of $(s-1)$-regular or nearly $(s-1)$-regular graphs on $n$ vertices that contain no path on $2s-1$ vertices. This family is non-empty whenever $n \ge 2s$ (see Proposition 2.1 in \cite{Y}).

Define the family $\mathcal{Y}(n,2s)$ as the collection of graphs formed by taking a complete bipartite graph with parts of sizes $n_1$ and $n_2$ (where $n_1\ge n_2$),  embedding a graph from $\mathcal{U}_{s, n_1}$ into the larger part, and placing a single edge in the smaller part. If
\[
n_1 \in \left\{ \left\lfloor\frac{2n + s - 1}{4}\right\rfloor, \left\lceil\frac{2n + s - 1}{4}\right\rceil \right\},
\]
then we refer to this family as $\mathcal{Y}'(n,2s)$.
Yuan \cite{Y} proved that $\ex(n,W_{2s}) = e(\mathcal{Y}'(n,2s))$ and $\EX(n,W_{2s})=\mathcal{Y}'(n,2s)$ for sufficiently large $n$. Observing the structure of $\mathcal{Y}'(n,2s)$, we note that every member of $\mathcal{Y}'(n,2s)$ is $\Wge{2s}$-free. Otherwise, suppose there is a member $M \in \mathcal{Y}'(n,2s)$  such that $W_{\ell} \subseteq M$ for some $\ell\ge 2s+1$. Let $u$ be the center vertex of $W_{\ell}$. If $u$ lies in the larger part, then its neighborhood inside the larger part has size at most $s-1$. Together with the single edge in the smaller part, this forces any cycle in $N_M(u)$ to have length at most $2s-1$, so $N_M(u)$ contains no copy of $C_{\ell}$. If $u$ lies in the smaller part, then a  cycle with length at least $2s+1$ in its neighborhood would give a path $P_{2s-1}$ in the larger part, which is impossible. Hence again $N_M(u)$ contains no $C_{\ell}$ for $\ell \geq 2s+1$. In either case we obtain a contradiction. Thus every member of $\mathcal{Y}'(n,2s)$ is $\Wge{2s}$-free. Consequently, we obtain the following result:

\begin{proposition}
	For sufficiently large $n$, $\ex(n,\Wge{2s}) = \ex(n,W_{2s})$.
\end{proposition}

Observe that $\mathcal{Y}(n,2s)$ is also $\Wge{2s}$-free. This suggests asking
whether $\EX(n,K_3,\Wge{2s})\subseteq \mathcal{Y}(n,2s)$ for sufficiently large $n$. Let $Y \in \mathcal{Y}(n,2s)$ with vertex partition $V(Y) = V(Y_1) \cup V(Y_2)$, where $n_i=|Y_i|$ and $n_1 \ge n_2$. Note that $e(Y_1) = \left\lfloor\frac{(s-1)n_1}{2}\right\rfloor$ with $n_1 = |Y_1|$.  This leads to determining whether the following asymptotic equality holds:
\[
	\ex(n, K_3, \Wge{2s}) = \max\left\{ \left\lfloor\frac{(s-1)n_1}{2}\right\rfloor n_2 : n_1 + n_2 = n, n_1\geq n_2 \right\} + o(n^2).
\]

\section*{Acknowledgment}
Part of this paper was finished when the second author was visiting Anhui University. The second author is grateful for the warm hospitality.

\end{document}